\documentclass[11pt]{amsart}
\usepackage{amssymb, adjustbox, amsbsy}
\usepackage{array, longtable, booktabs, tabularx}
\usepackage{geometry}

\newcolumntype{Y}{>{\centering\arraybackslash}X}

\usepackage{amsfonts, amscd}
\numberwithin{equation}{section}

\usepackage{bm}
\usepackage{verbatim}
\usepackage{graphicx}
\graphicspath{{images/}}
\usepackage{tikz, tikz-cd}
\usepackage{subcaption}
\usepackage{listings}
\usepackage{subfiles}
\usepackage{mathtools}
\usepackage{comment}
\usepackage{enumitem}
\usepackage[all]{xy}
\usepackage{hyperref}
\hypersetup{
    colorlinks=true,
    citecolor=red,
    linkcolor=blue,
    filecolor=magenta,      
    urlcolor=red,
}

\usepackage[textsize=footnotesize,bordercolor=red,linecolor=red,backgroundcolor=red!15]{todonotes}

\newcommand{\PP}{\mathbb{P}}
\newcommand{\QQ}{\mathbb{Q}}

\newcommand{\proj}{\operatorname{Proj}}

\newcommand{\mult}{\operatorname{mult}}

\newtheorem{thm}{Theorem}[section]

\newtheorem{lem}[thm]{Lemma}

\theoremstyle{definition}

\newtheorem{ques}[thm]{Question}
\newtheorem{rem}[thm]{Remark}
\newtheorem{ex}[thm]{Example}

\newcommand{\ord}{\mathrm{ord}}

\tikzstyle{wbullet}=[circle, draw=black, fill=white, thick, inner sep=2pt, minimum size=1.5mm]
\tikzstyle{bbullet}=[circle, draw=black, fill=black, inner sep=2pt, minimum size=1.5mm]

\begin{document}

\title{On a question of Mauri and Moraga}

\author{Jihao Liu}

\address{Department of Mathematics, Peking University, No. 5 Yiheyuan Road, Haidian District, Beijing 100871, China}
\address{Beijing International Center for Mathematical Research, Peking University, No. 5 Yiheyuan Road, Haidian District, Beijing 100871, China}
\email{liujihao@math.pku.edu.cn}

\date{\today}

\begin{abstract}
We give negative answers to both parts of a question of Mauri and Moraga on log Calabi-Yau pairs whose boundary decomposes into big divisors. The main result of this paper is obtained by generative AI, particularly Chatgpt 5.5 pro and the Rethlas system.
\end{abstract}

\maketitle

\section{Introduction}\label{sec:introduction}

We work over the field of complex numbers $\mathbb C$. 

Mauri and Moraga asked the following question \cite[Question~9.9]{MM25}.
\begin{ques}[{\cite[Question~9.9]{MM25}}]\label{ques:mm25-9.9}
Let $(X,B)$ be an $n$-dimensional log Calabi-Yau pair such that $B$ admits a decomposition into $(n+1)$ big divisors. 
\begin{enumerate}
    \item Is the pair $(X,B)$ qdlt?
    \item Is $\mathcal{D}(X,B)$ isomorphic to the boundary of an $n$-dimensional standard simplex?
\end{enumerate}
\end{ques}
By \cite[Theorem~1.10]{MM25}, the pair $(X,B)$ is crepant birational to a weighted projective space and $\mathcal{D}(X,B)$ is PL-homeomorphic to a PL-sphere. This provides strong evidence for a positive answer to Question~\ref{ques:mm25-9.9}. However, both parts of Question~\ref{ques:mm25-9.9} have negative answers.

\begin{thm}\label{thm:main-qdltness}
There exists a projective three-dimensional log canonical log
Calabi-Yau pair (as in Example~\ref{ex:counter-mm25-9.9-qdltness}) $\left(X,B=\sum_{i=1}^4D_i\right)$
such that each $D_i$ is a prime big divisor, but $(X,B)$ is not qdlt.
In particular, Question~\ref{ques:mm25-9.9}(1) has a negative answer. 
\end{thm}
\begin{thm}\label{thm:main-dual-complex}
There exists a surface qdlt log Calabi-Yau pair $\left(S,\Delta=\sum_{i=1}^3C_i\right)$ such that each $C_i$ is a big divisor but $\mathcal{D}(S,\Delta)$ is not isomorphic to the boundary of the two-dimensional standard simplex. In particular, Question~\ref{ques:mm25-9.9}(2) has a negative answer.
\end{thm}

\begin{rem}
Both parts of Question~\ref{ques:mm25-9.9} are obtained by generative AI. Question~\ref{ques:mm25-9.9}(2) is a one-shot by Chatgpt 5.5 pro. Question~\ref{ques:mm25-9.9}(1) is an almost one-shot by Chatgpt 5.5 pro, but it messed up on showing that the singularity
$$(v^2-u^3-\tau s=0)\subset\mathbb A^4$$
is not a quotient singularity. The paper was then forwarded to the Rethlas system, which solved the problem via a very lengthy proof. The current proof of the corresponding part just refers to \cite{Rei87} as the classification of threefold cDV quotient singularities is well-known. 

See \cite{Ju+26} for a detailed introduction to the Rethlas system. Due to the limitation of generative AI, it is possible that we have missed some related references in the literature, and we welcome any comments from experts.
\end{rem}

\subsection*{Acknowledgements}
The author was partially supported by the National Key R\&D Program of China \#\allowbreak 2024YFA1014400. The author would like to thank the Rethlas team, namely Haocheng Ju, Jiedong Jiang, Shurui Liu, Guoxiong Gao, Yuefeng Wang, Zeming Sun, Bin Wu, Liang Xiao, and Bin Dong, for their contributions to the development of Rethlas and its customized version used for the problem studied in this paper. The author would like to thank Kaiyuan Gu, Ruicheng Hu, and Sheng Qin for assistance with the verification of an earlier blueprint of this paper. The author would like to thank Ruochuan Liu and Gang Tian for constant support and encouragement.

\section{The examples}

In this section, we construct the examples proving Theorems~\ref{thm:main-qdltness} and~\ref{thm:main-dual-complex}.

\begin{ex}\label{ex:counter-mm25-9.9-qdltness}
Let
\[
    T=\PP(1,1,1,4)=\proj \mathbb C[x,y,z,t],
    \qquad
    \deg x=\deg y=\deg z=1,
    \quad
    \deg t=4.
\]
Let
\[
    D_x=(x=0),\qquad D_y=(y=0),\qquad
    D_z=(z=0),\qquad \text{and}\qquad D_t=(t=0)
\]
be the toric boundary divisors, and put
\[
    B_T=D_x+D_y+D_z+D_t.
\]
The unique singular point of $T$ is
\[
    o=[0:0:0:1],
\]
of quotient type $\frac{1}{4}(1,1,1)$. Since $o\notin D_t$, the divisor
$D_t$ lies in the smooth locus of $T$, and
\[
    D_t\simeq \PP^2_{x,y,z}.
\]
Inside $D_t\simeq \PP^2$, take the cuspidal cubic
\[
    C=(h=0),\qquad h=y^2z-x^3.
\]
Let
\[
    Z=(t=h=0)\subset T.
\]
Thus $Z=C\subset D_t$. The curve $Z$ is irreducible and is a
codimension-two local complete intersection in the smooth locus of $T$.
It is singular only at the cusp
\[
    p=[0:0:1:0].
\]
Let
\[
    \pi\colon X=\operatorname{Bl}_Z T\longrightarrow T
\]
be the blow-up of $T$ along $Z$, and let $E$ be the exceptional divisor of $\pi$.
Define
\[
    B=D_1+D_2+D_3+D_4,
\]
where $D_1,D_2,D_3,D_4$ denote the strict transforms of
$D_x,D_y,D_z,D_t$ on $X$ respectively.
\end{ex}

\begin{lem}\label{lem:threefold-lcy}
The pair $(X,B)$ in Example~\ref{ex:counter-mm25-9.9-qdltness} is log
canonical and satisfies
\[
    K_X+B\sim_{\QQ}0.
\]
\end{lem}

\begin{proof}
The ideal of $Z$ is locally generated by the regular sequence $(t,h)$. Since $\pi$ is the blow-up of a codimension-two regular closed subset, we have
\[
    K_X=\pi^*K_T+E.
\]
Since $Z\subset D_t$ with generic multiplicity one, while $Z$ is not
contained in $D_x,D_y$, or $D_z$, we have
\begin{equation}\label{eq:pullback}
 \pi^*D_t=D_4+E,\quad   \pi^*D_x=D_1,\quad     \pi^*D_y=D_2,\quad \text{and}\quad    \pi^*D_z=D_3.   
\end{equation}
Thus
\[
K_X+B=\pi^*(K_T+B_T)\sim 0
\]
as $(T,B_T)$ is log toric and $K_T+B_T\sim 0$. Since $(T,B_T)$ is lc, $(X,B)$ is lc.
\end{proof}

\begin{lem}\label{lem:threefold-singularities}
Notation as in Example~\ref{ex:counter-mm25-9.9-qdltness}. Then $X$ is klt and has two isolated singularities $q$ and $\pi^{-1}(o)$, where $q$ is locally analytically isomorphic to
\[
\{v^2-u^3-\tau s=0\}\subset \mathbb A^4_{u,v,\tau,s}.
\]
In particular, $q$ is a cDV singularity of type cA$_2$ and is not a quotient singularity.
\end{lem}

\begin{proof}
Work on the affine chart $z=1$ around
\[
    p=[0:0:1:0].
\]
Set
\[
    u=x/z,\qquad v=y/z,\qquad \tau=t/z^4.
\]
Then the affine chart is
\[
    \mathbb A^3_{u,v,\tau},
\]
and $Z$ is defined by
\[
    I=(\tau,f),\qquad f=v^2-u^3.
\]
The blow-up of $I$ has two affine charts $U_{\tau}$ and $U_f$. In the $U_{\tau}$ chart, put
\[
    s=\frac{f}{\tau}.
\]
Then we have
\[
    U_\tau=\{v^2-u^3-\tau s=0\}\subset \mathbb A^4_{u,v,\tau,s}.
\]
By the classification of threefold terminal singularities (cf.~\cite[(6.7), Proposition]{Rei87}), $U_\tau$ has a unique cDV singularity $q$ of type cA$_2$ and is not a quotient singularity.
\end{proof}

\begin{lem}\label{lem:threefold-big}
Notation as in Example~\ref{ex:counter-mm25-9.9-qdltness}. Then $D_i$ is big for any $1\leq i\leq 4$.
\end{lem}
\begin{proof}
Let $H:=\mathcal O_T(1)$. By \eqref{eq:pullback}, 
\[
D_1\sim_{\mathbb Q} D_2\sim_{\mathbb Q} D_3\sim_{\mathbb Q}\pi^*H
\]
is big. Let $G:=(h=0)\subset T$. Then $G\sim_{\mathbb Q}3H$. Since
\begin{equation}
Z=D_t\cap G\qquad \text{and}\qquad  \mult_ZD_t=\mult_ZG=1,
\end{equation}
we have
\[
    D_4\sim_{\QQ}\pi^*(4H)-E
    \qquad \text{and}\qquad 
    \pi^{-1}_*G\sim_{\QQ}\pi^*(3H)-E,
\]
so $D_4\sim_{\mathbb Q}\pi^*H+\pi^{-1}_*G$ is big. 
\end{proof}

\begin{lem}\label{lem:q-lc-center}
Notation as in Example~\ref{ex:counter-mm25-9.9-qdltness} and Lemma~\ref{lem:threefold-singularities}. Then $q$ is an lc center of $(X,B)$.
\end{lem}

\begin{proof}
On the local coordinate chart $u,v,\tau$ as in Lemma~\ref{lem:threefold-singularities}, we define
\[
    Y=\mathbb A^3_{u,v,\tau}
    \qquad \text{and}\qquad 
    \Delta=(u=0)+(v=0)+(\tau=0).
\]
Let $F$ be the divisor corresponding to the ordinary blow-up of $Y$ at $(0,0,0)$. Then $F$ is an lc place of $(Y,\Delta)$ as $(Y,\Delta)$ is snc. By considering $F$ as a divisor over $T$, we have that $F$ is an lc place of $(T,B_T)$, hence an lc place of $(X,B)$ as $K_X+B=\pi^*(K_T+B_T)$. Now consider $F$ as a divisor over $X=\operatorname{Bl}_{(\tau,v^2-u^3)}Y$. We have
\[
    \ord_F(u)=\ord_F(v)=\ord_F(\tau)=1
    \qquad \text{and}\qquad
    \ord_F(v^2-u^3)=2.
\]
On the $\tau$-chart of $X$,
\[
    s=\frac{v^2-u^3}{\tau},
\]
so
\[
    \ord_F(s)=\ord_F(v^2-u^3)-\ord_F(\tau)=1.
\]
Thus all local coordinates $u,v,\tau,s$ have positive $F$-order, and the
center of $F$ on the $\tau$-chart is the closed point
\[
    q=(u,v,\tau,s)=(0,0,0,0).
\]
Thus $q$ is an lc center of $(X,B)$.
\end{proof}

\begin{proof}[Proof of Theorem~\ref{thm:main-qdltness}]
By Lemma~\ref{lem:q-lc-center}, $q$ is an lc center of $(X,B)$. Since the generic point of an lc center of any qdlt pair is a quotient singularity by \cite[Proposition~34 and Definition~35]{dFKX17}, $(X,B)$ is not qdlt by Lemma~\ref{lem:threefold-singularities}.
\end{proof}

We also prove Theorem~\ref{thm:main-dual-complex} by considering a surface construction analogous to Example~\ref{ex:counter-mm25-9.9-qdltness}.

\begin{proof}[Proof of Theorem~\ref{thm:main-dual-complex}]
Let
\[
    T=\PP(1,1,2)=\proj \mathbb C[x,y,z],
    \qquad
    \deg x=\deg y=1,
    \quad
    \deg z=2.
\]
Let
\[
    L_x=(x=0),\qquad L_y=(y=0),\qquad L_z=(z=0)
\]
be the toric boundary divisors. Let $h\colon S\rightarrow T$ be the blow-up of $L_x\cap L_z$ with exceptional divisor $E$. Let
\[
C_1=h^{-1}_*L_z+E,\quad C_2:=h^{-1}_*L_x,\quad\text{and}\quad C_3:=h^{-1}_*L_y.
\]
Then $(S,\Delta:=\sum_{i=1}^3C_i)$ is a surface qdlt log Calabi-Yau pair. But $\mathcal{D}(S,\Delta)$ is a square and hence is not isomorphic to the boundary of a two-dimensional standard simplex.  
\end{proof}

\end{document}